\documentclass[a4paper, 11 pt, twoside]{amsart}
\usepackage{srollens-en, shortcuts-sr}
\usepackage{todonotes}
\usepackage[left = 3.5cm, right = 3.5cm, headsep = 6mm,
footskip = 10mm, top = 35mm, bottom = 35mm, footnotesep=5mm, headheight =
2cm]{geometry}

\newlist{prooflist}{description}{1}
\setlist[prooflist]{font=\normalfont \itshape, labelindent = \parindent, leftmargin = 0pt}

\usepackage{comment}

\usepackage[unicode,bookmarks, pdftex]{hyperref}
\hypersetup{colorlinks=true,citecolor=NavyBlue,linkcolor=NavyBlue,urlcolor=Orange, pdfpagemode=UseNone, breaklinks=true}

\newcommand{\sr}[1]{\textsf{\scriptsize\color{Plum} SR: {#1}}}

\newcommand{\Sp}{\mathsf{sp}}

\newcommand{\I}{\mathrm{I}}
\newcommand{\DI}{\mathrm{II}}
\newcommand{\III}{\mathrm{III}}

\title{Yet another family of complex structures on $S^6$}

\author{Wenfei Liu}
\address{School of Mathematical Sciences, Xiamen University, Siming South Road 422, 361005 Xiamen, Fujian Province, P.~R.~China}
\email{wliu@xmu.edu.cn}

\author{S\"onke Rollenske}
\address{FB 12/Mathematik und Informatik\\
Philipps-Universit\"at Marburg\\Hans-Meerwein-Str. 6\\
35032 Marburg\\Germany}

\email{rollenske@mathematik.uni-marburg.de}

\begin{document}
\begin{abstract}
Following a similar path as the recent construction of a complex structure on $S^6$ we construct another such  family, disjoint from the known ones. 
\end{abstract}


\maketitle
\setcounter{tocdepth}{1}
\tableofcontents

\section{Introduction}
The question of which spheres admit complex structures, that is, can be endowed with a holomorphic atlas to become compact complex manifolds, 
was raised by Hopf in \cite{Hopf1948}, where he pointed out that some manifolds, among them $S^4$ and $S^8$ cannot admit an almost complex structure.
After partial results by many authors, Borel and Serre proved in 
\cite{BorelSerre} that $S^2$ and $S^6$ are the only spheres admitting almost complex structures. Since $S^2$ can be identified with the complex projective line, the remaining 
\emph{ Hopf problem} was to prove or disprove the existence of a complex structure on $S^6$. Over the next decades there were several attempts to settle this, compare 
\cite{surveyS6} for the state of the art in 2017, and  \cite{zbMATH07167603,  zbMATH08259842} for recent contributions, but it  was still included in Yau's list of open problems in geometry \cite{Yau2024}.

Recently, contrary to the expectations of many\footnote{One obstacle was that, due to  a gap in the literature, the possibility of a complex structure of algebraic dimension $1$ on $S^6$, was known only to a few experts. We learned about this only after \cite{Alpoge2026} appeared, from J.\ Viaclovsky.}, a 1-parameter family of complex structures on $S^6$ was constructed by AI 
\cite{Alpoge2026}. Subsequently, Engel distilled the geometric essence of the construction in  \cite{Engel2026}.

The present paper grew from the attempt to understand that construction, by performing its analogue in a related setting. More due to luck than systematic explorations we found another family of complex structures on the sphere $S^6$.
\begin{custom}[Theorem \ref{thm: main}]
	There exists a  family of compact  complex manifolds $f\colon Y(u) \to \IP^1$ parametrised by  $u$ in a sufficiently small punctured disk such that the differentiable manifold underlying $Y(u)$ is diffeomorphic to $S^6$ and which is disjoint from the family constructed in \cite{Alpoge2026, Engel2026}.
\end{custom}
For the sake of simplicity, we decided not to explore what happens if we vary the discrete parameter chosen in Definition \ref{def: choice of a}. Having stumbled upon a second family of complex structures on $S^6$ and then being informed about an independent construction by Jeff Viaclovsky (see Remark \ref{rem: third family} below) we wonder if there are several more out there or if all constructions of this kind will turn out to belong to one connected family. 

Our proof follows the approach of Engel very closely, as do several other recent preprints \cite{GuoFangLu2026S2xS4,Yang2026CompactComplexThreefolds,Wang2026RationalHomologySixSphere, LiangLiuZhang2026ComplexStructures, Chen2026ConditionalOka, Viaclovsky2026TwoParameterS6}. Beyond the existence of another family, we hope that the present paper can be a useful companion to \cite{Engel2026}, because we organise the material more to our taste and thus emphasise different aspects.  In contrast  to some recent contributions we focus on human readability throughout.

\begin{rem}\label{rem: third family}
After the completion of this paper we were informed by Jeff Viaclovsky that he has constructed a 2-dimensional family of complex structures on $S^6$ from a rational elliptic surface with one fibre of type $\III^*$ and three fibres of type $\I_1$ \cite{Viaclovsky2026TwoParameterS6}. Since our surface is a degeneration of his surface, we strongly suspect that our family is a degeneration of his family. 
The families are still disjoint, by the same argument used in Theorem \ref{thm: main}.

A computer search by Engel  revealed the existence of at least seventeen further families of complex structures on $S^6$ \cite{EngelThreefoldExplorer}. 
In addition, countably many examples were constructed in \cite{viaclovsky2026combscomplexstructuressixsphere} by some kind of surgery on the orginal example of \cite{Alpoge2026}. 
\end{rem}

\subsection*{Acknowledgements} 
The second author had the pleasure to visit Yongnam Lee at IBS in Daejeon while working through \cite{Engel2026} and is grateful for the stimulating atmosphere and discussions. 
He would also like to thank Giovanni Bazzoni, Nicolina Istrati and Jeff Viaclovsky for  discussions on the topic and Philip Engel for some email exchange. 

Wenfei Liu was supported by the NSFC (no.~12571046). Sönke Rollenske was supported by the DFG.   We acknowledge the financial support for this collaboration provided by the Tianyuan Mathematical Center in Southeast China.

\subsection{AI workflow and use disclosure}\label{sect: AI use}
Since there is currently no established standard for what is considered good and ethical practice in the use of AI, we give a short description of our workflow.

We explored the construction of \cite{Engel2026} starting with the surface $S$ in Section \ref{sect: S} first by hand, then using ChatGPT \cite{openai2026chatgpt} to carry the burden of calculations.

Following the approach verbatim gives a rational homology sphere with fundamental group $\IZ/2$, where the non-contractible loop comes from the fact that local degrees of the good reduction models, $4$ at the type $\III^*$ fibre and $6$ at the type $\DI$ fibre have a common factor of $2$. Reducing the multiplicity of the log-transform at the type $\DI$ fibre to $3$ gives indeed a simply connected threefold, which however has $2b_2=b_3>0$, with classes coming from components of the now singular and reducible fibre at $1$. The next try was to perform the log-transform on a smooth fibre instead, which did not help at first until we asked if any choice of  multiplicity might lead to a $6$-sphere. At that point multiplicity zero was suggested, yielding the current construction. 

In the writing process we prescribed the structure of the paper, stating propositions, some of them already proven in our notebooks or on our blackboard, and then letting ChatGPT fill in step by step, closely monitoring notation and readability. The final version has seen several cycles of human revision and additions, then possibly AI filling in the computational parts again. Being able to globally change notations easily at a later stage was extremely beneficial for the final presentation.

The introduction was entirely written by us. We hope that the result is an interesting and  useful, human-readable paper.

\subsection{Notations}\label{sec: notation}
We will use the following notations:
\begin{itemize}
 \item $\Delta_p$ will denote a disc in $\IP^1$ centered at the point $p$. We will shrink it or perform a local analytic coordinate change when necessary. 
\item For a basis $(e_i)$ of a finite free module, we denote the dual
basis by $(e_i^*)$, so $e_i^*(e_j)=\delta_{ij}$, and write
$ e_{i_1\cdots i_k}^*=e_{i_1}^*\wedge\cdots\wedge e_{i_k}^*$ for exterior products.

\item Given a fibration  $\pi \colon S \to B$ we denote by $S_b$ the fibre over $b\in B$ and by $S_U$ the induced fibration $\inverse \pi U \to U$ for an open subset $U\subset B$. 
\item In computations, the symbol $I$ is always an identity matrix or map of the required size.
\end{itemize}

\section{Construction of the family}

\subsection{The underlying rational elliptic surface and line bundle}\label{sect: S}

Let $\pi\colon S\to\IP^1$
be the rational elliptic surface with zero
section $O$ and
singular fibres
\[
S_0\text{ of type }\III^*,\qquad
S_1\text{ of type }\DI,\qquad
S_\infty\text{ of type }\I_1,
\]
schematically depicted in Figure \ref{fig: rational elliptic surface}.
 By
\cite[Main Theorem, No.~43]{OguisoShioda1991}, 
the Mordell--Weil
group
is infinite cyclic.  We choose
a generator 
$P$ and set
$M=\ko_S(P-O)$.
\begin{figure}
\centering
\begin{tikzpicture}[x=1cm,y=1cm,line cap=round,line join=round,
 fibre/.style={draw=blue!70!white,line width=.85pt},
 mult/.style={font=\scriptsize,text=blue!65!black,inner sep=1pt}]
 \draw[gray!80,line width=.8pt]
   (0,4.45) .. controls (.25,3.2) and (-.12,1.8) .. (0,.65);
 \draw[gray!80,line width=.8pt]
   (10.2,4.4) .. controls (10,3.1) and (10.05,1.8) .. (10.2,.6);
 \draw[gray!65,line width=1pt]
   (4.2,4.35) .. controls (3.85,3.5) and (4.55,3.1) .. (4.2,2.5)
   .. controls (3.9,1.9) and (3.95,1.2) .. (4.2,.7);
 \draw[fibre] (1.65,4.6)--(2.4,3.85);
 \draw[fibre,line width=1.1pt] (1.65,3.35)--(2.4,4.1);
 \draw[fibre,line width=1.4pt] (1.65,3.6)--(2.4,2.85);
 \draw[fibre,line width=1.7pt] (1.65,2.35)--(2.4,3.1);
 \draw[fibre,line width=1.4pt] (1.65,2.6)--(2.4,1.85);
 \draw[fibre,line width=1.1pt] (1.65,1.35)--(2.4,2.1);
 \draw[fibre] (1.65,1.6)--(2.4,.85);
 \draw[fibre,line width=1.1pt] (1.1,2.75)--(2.65,2.75);
 \foreach \x/\y in {2.275/3.975,1.775/3.475,2.275/2.975,
                       1.775/2.475,2.275/1.975,1.775/1.475,2.05/2.75}
   \filldraw[draw=blue!70!black,fill=cyan!45,line width=.4pt]
     (\x,\y) circle (1.35pt);
 \node[mult] at (1.65,4.2) {$1$};
 \node[mult] at (2.55,3.72) {$2$};
 \node[mult] at (1.5,3.25) {$3$};
 \node[mult] at (2.55,2.48) {$4$};
 \node[mult] at (1.5,2.2) {$3$};
 \node[mult] at (2.55,1.73) {$2$};
 \node[mult] at (1.9,1.08) {$1$};
 \node[mult,above] at (1.27,2.78) {$2$};
 \node[blue!70!white] at (.72,1.7) {$\III^*$};
 \draw[fibre] (6.4,4.5)
   .. controls (5.9,2.8) and (5.65,2.8) .. (5.65,2.8)
   .. controls (5.65,2.8) and (5.9,2.8) .. (6.4,.8);
 \node[blue!70!white] at (5.65,2.3) {$\DI$};
 \draw[fibre] (9.3,4.6)
   .. controls (9.5,3.85) and (9.5,3.1) .. (9.05,2.6)
   .. controls (8.3,1.8) and (7.65,2.15) .. (8,2.8)
   .. controls (8.3,3.3) and (8.8,2.95) .. (9.05,2.6)
   .. controls (9.5,2) and (9.55,1.4) .. (9.3,.75);
 \filldraw[draw=blue!70!black,fill=cyan!45,line width=.4pt]
   (9.05,2.6) circle (1.5pt);
 \node[blue!70!white] at (8.55,1.9) {$\I_1$};
 \draw[red!75!black,line width=1.15pt]
   (0,4.45) .. controls (.65,4.3) and (1.1,4.6) .. (1.65,4.6)
   .. controls (2.7,4.6) and (3.3,4.25) .. (4.2,4.35)
   .. controls (5,4.45) and (5.6,4.6) .. (6.4,4.5)
   .. controls (7.5,4.3) and (8.4,4.65) .. (9.3,4.6)
   .. controls (9.65,4.58) and (9.95,4.5) .. (10.2,4.4);
 \draw[green!45!black,line width=1.15pt]
   (0,.65) .. controls (.9,.45) and (1.6,.85) .. (2.4,.85)
   .. controls (3,.85) and (3.65,.6) .. (4.2,.7)
   .. controls (5,.85) and (5.6,.9) .. (6.4,.8)
   .. controls (7.5,.55) and (8.5,.9) .. (9.3,.75)
   .. controls (9.65,.72) and (10,.65) .. (10.2,.6);
 \node[red!75!black,below] at (7.3,4.4) {$O$};
 \node[green!45!black,above] at (7.3,.65) {$P$};
 \node[right] at (10.3,2.6) {$S$};

\draw[->] (5,0.25) to node[right]{$\pi$} ++(0,-.75);

\begin{scope}[yshift = -1cm]
	 \draw[gray!80,line width=.8pt] (.05,-.25)--(10.2,-.25);
 \foreach \x/\lab in {2/0,6.4/1,9.05/\infty} {
   \fill[blue!70!white] (\x,-.25) circle (2pt);
   \node[above] at (\x,-.2) {$\lab$};
 }
 \fill[gray!65] (4.2,-.25) circle (2pt);
 \node[above] at (4.2,-.2) {$t$};
 \node[right] at (10.25,-.25) {$\IP^1$};
 \end{scope}
\end{tikzpicture}
\caption{The rational elliptic surface $S$. Numbers indicate the
component multiplicities of the $\III^*$-fibre.}
\label{fig: rational elliptic surface}
\end{figure}
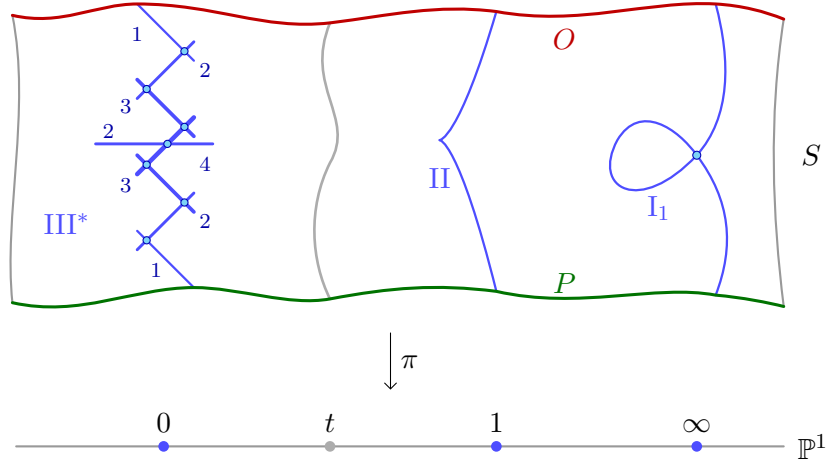
We give an algebraic model and collect some properties needed
in the construction.
\begin{prop}\label{prop: surface facts}
	The surface $S$
and 
the section $P$ can be given
by
	\begin{equation}\label{eq:weierstrass-S}
		y^2=x^3+\frac{27}{4}t^3(1-t)x
		+\frac{27}{4}t^5(1-t),
		\qquad P=(-t^2,\mathrm{i}t^3),
	\end{equation}
	where $O$ is the section at infinity.  Moreover:
\begin{enumerate}
		\item
		The discriminant is $\Delta=-3^9t^9(1-t)^2$, the $j$-function is $j(t)=1728(1-t)$.
		
		\item The section
$P$ meets a non-identity component of
the fibre
		of type $\III^*$.
		
		\item There is an isomorphism $\shom_{\ko_S}(t_{2P}^*M,M)		\isom\pi^*\ko_{\IP^1}(1)$.
	
\end{enumerate}
\end{prop}

\begin{proof}
	The formulas for $\Delta$ and $j$ follow directly from
	\eqref{eq:weierstrass-S}. The vanishing orders of the Weierstrass
	coefficients and the discriminant are $(3,5,9)$ at $t=0$ and
	$(1,1,2)$ at $t=1$.  Kodaira's classification
	\cite[Chapter~IV, \S3]{Miranda1989} therefore gives fibres of types
	$\III^*$ and $\DI$, respectively.

	At infinity, put $v=t^{-1}$, $X=v^2x$, and $Y=v^3y$. Then the equation becomes $Y^2=X^3+\frac{27}{4}(v-1)X+\frac{27}{4}(v-1)$,
	whose discriminant has a simple zero at $v=0$. Thus the fibre at
	infinity has type $\I_1$.

	The section $P$ is disjoint from $O$ and specializes at $t=0$ to the
	singular point $(0,0)$ of the minimal Weierstrass model. It therefore
	meets a non-identity component of the $\III^*$-fibre. Its local height
	contribution is $3/2$ by
	\cite[Theorem~8.6 and Table~8.16]{Shioda1990}, 
    so
\[	\langle P,P\rangle
	=2+2(P\cdot O)-\frac32
	=\frac12.
	\]
	This agrees with the fact that a generator of the Mordell--Weil
	lattice $A_1^*$ has height $1/2$ (\cite[Table~6.10]{Shioda1990}).

The section $Q:=2P$ meets the identity component of the $\III^*$-fibre.
	Hence all its local height contributions vanish, and
$2=\langle Q,Q\rangle=2+2(Q\cdot O)$.
	Thus $ Q\cdot O=0$. Similarly, $1=\langle P,Q\rangle=1-P\cdot Q$,
	so $P\cdot Q=0$.

	Set
	\[
	L=\shom_{\ko_S}(t_{Q}^*M,M).
	\]
	On a smooth fibre, translation acts trivially on degree-zero line
	bundles. Moreover, $Q$ preserves every component of the reducible
	fibre. Consequently $L$ is vertical and has degree zero on every
	fibre component. Therefore
$L\isom\pi^*\ko_{\IP^1}(n)$
	for some $n\in\IZ$. Restricting to the zero section gives
	\[
	n=L\cdot O
	=(O-P)\cdot Q-(O-P)\cdot O
	=1.
	\]
	This proves the last assertion.
\end{proof}
\subsection{The open family $X(u)^\circ$}
Let
$M^\times$ 
be the complement of the zero section in $M$
and put
\[
	S^\circ=S\setminus S_\infty,
	\qquad
	M^{\times,\circ}=M^\times|_{S^\circ}.
\]
Thus $\pi^\circ\colon S^\circ\to\IC$ is
the restriction of $\pi$
to
the affine line $\IP^1\setminus\{\infty\}$.

By Proposition~\ref{prop: surface facts},
there is, up to a non-zero
scalar, a unique homomorphism
$\phi\colon t_{2P}^*M\longrightarrow M$
 over $S$,
whose divisor is $S_\infty$, thus its restriction to $S^\circ$ is an
isomorphism.  For $u\in\IC^*$ we consider the
$\IZ$ action 
on
$M^{\times,\circ}$  generated by
\begin{equation}\label{eq: Z action}
	\tau(s,m)=
	\left(t_{-2P}(s),u\phi(t_{2P}^*m)\right).
\end{equation}
Here,
for $m\in M_s$,
$t_{2P}^*m$ denotes the same $m$, but
viewed as an element of
$(t_{2P}^*M)_{t_{-2P}(s)}=M_s$.
\begin{prop}\label{prop: construction X}
For every sufficiently small $u\in\IC^*$
the following hold.
\begin{enumerate}
	\item The action~\eqref{eq: Z action}
is free and
properly
	discontinuous.  The
quotient
	\[
		X(u)^\circ=M^{\times,\circ}/\langle\tau\rangle
	\]
is a complex manifold, and the induced map
	$\eta^\circ\colon X(u)^\circ\to\IC$ is proper.  It fits into
the
	diagram
	\[
	\begin{tikzcd}[ampersand replacement=\&]
		M^{\times,\circ}\rar{/\langle\tau\rangle}\dar
		\&X(u)^\circ\dar{\eta^\circ}\\
		S^\circ\rar{\pi^\circ}\&\IC.
	\end{tikzcd}
	\]
\item Over $\IC\setminus\{0,1\}$, the map $\eta^\circ$
is a smooth
	family of
compact complex two-dimensional tori.
\item The map $\eta^\circ$ admits
a holomorphic
section
	$\sigma^\circ$.  
\end{enumerate}
\end{prop}

\begin{proof}
By Proposition~\ref{prop: surface facts}, the linearization has the
same properties as in \cite[Section~2]{Engel2026}, with $2P$ in place
of $6P$.  The contracting quotient construction there gives a proper
fibration with smooth total space; freeness and proper discontinuity
follow as in \cite[Proposition~2.4]{Engel2026}.

On each smooth fibre $S_t$, the bundle $M_t$ has degree zero and
$\tau$ lifts translation by $-2P$.  Thus
\cite[Proposition~2.6]{Engel2026} applies and gives (ii).

Finally,
\[
	\deg(M|_O)=(P-O)\cdot O=1,
\]
so $M|_O\simeq\ko_{\IP^1}(1)$.  Choose a section $e$ of $M|_O$ whose
only zero is at $\infty$.  Over $\IC$, it gives a section of
$M^{\times,\circ}\to\IC$, and its image in the quotient defines
$\sigma^\circ$.  
\end{proof}

From here on we fix a suitable $u$, sufficiently small for all
the constructions below, and drop it from the notation, as it will
play no further role.

\subsection{Compactification over the $\I_1$-fibre---Mumford construction}

The compactification at infinity is the same as in
\cite[Propositions~2.7--2.9]{Engel2026} and we defer to the expert for some details.

\begin{prop}\label{prop: compactification X}
The fibration $\eta^\circ\colon X^\circ\to\IC$ extends to a proper
map $\eta\colon X\to\IP^1$ from a smooth compact complex threefold.
The fibre $X_\infty$ is obtained from a del Pezzo surface of degree six
by identifying opposite sides of its anticanonical hexagon.  The
section $\sigma^\circ$ extends to a holomorphic section $\sigma$ of
$\eta$.
\end{prop}

\begin{proof}
The fibre $S_\infty$ has type $\I_1$, the bundle $M$ has degree zero
on it, and $2P$ specializes to its smooth locus.  Moreover, $\phi$
vanishes to first order along $S_\infty$.  Thus the hypotheses of
\cite[Proposition~2.7]{Engel2026} hold, with $2P$ in place of $6P$.
For a local coordinate $q$ at infinity, the fibres over a punctured
disk therefore have the form
\[
 X_q\isom (\IC^*)^2/
 \langle(c_1q,c_2),(c_3,c_4q)\rangle,
\]
where the $c_i$ are holomorphic units on the disk and the displayed
pairs act by coordinatewise multiplication.

We use the $\IZ^2$-periodic triangulation obtained by dividing each
unit square along the same diagonal.  The Mumford construction of
\cite[Proposition~2.8]{Engel2026} then gives the required filling,
with smooth total space and central fibre as stated.  Gluing it to
$X^\circ$ gives $X$; since $\eta$ is proper, $X$ is compact.

Finally, $\sigma^\circ$ was constructed from a section of $M|_O$
with a simple zero at infinity.  Its extension across the Mumford
filling follows exactly as in \cite[Proposition~2.9]{Engel2026}.
\end{proof}
\subsection{Logarithmic transformation at the $\III^*$-fibre}
We follow \cite[Section~3]{Engel2026}, using the local model of
\cite[V.10]{BHPV2004}. After a suitable analytic coordinate change
at $0$, we get an explicit description of  the good reduction model after a cyclic base change,
\[ 
\begin{tikzcd}
	 \tilde S_{\Delta_0} \rar[dashed] \dar & S_{\Delta_0}\dar\\
	 \tilde \Delta_0 \rar{s\mapsto s^4} & \Delta_0 
\end{tikzcd}
\]
as the elliptic surface\footnote{In the normal form the exponent of $s$ in $z(s)$ is determined as follows: the function $j-1728$ has a simple zero on $\Delta_0$, hence a zero of order four after base change.  Since the modular
$j$-function has ramification index two at $i$, the period $z(s)-i$
has a zero of order two.}
\[
 \tilde S_{\Delta_0}=(\IC\times\tilde\Delta_0)/(\IZ+z(s)\IZ)
 \qquad
 z(s)=i\frac{1+s^2}{1-s^2}.
\]
The deck group is generated by
\begin{equation}\label{eq: rho}
 \rho(c,s)=\left(\frac{c}{z(s)},is\right).
\end{equation}
Indeed, $z(is)=-1/z(s)$, so this preserves the period lattices and
defines an action of order four 
 on $\tilde S_{\Delta_0}$.

On $\tilde S_0=\IC/(\IZ+i\IZ)$, the action is multiplication by $-i$.
Its fixed points are $0$ and $(1+i)/2$, while $1/2$ and $i/2$ form
one orbit with stabilizer of order two.  Thus
$\tilde S_{\Delta_0}/\langle\rho\rangle$ has two $A_3$ singularities and one
$A_1$ singularity and its minimal resolution is $S_{\Delta_0}$.

\begin{lem}\label{lem: good reduction at 0}
Let $\tilde P$ and $\tilde O$ be the transforms of $P$ and $O$ on
$\tilde S_{\Delta_0}$, and let $\tilde M=\ko_{\tilde S_{\Delta_0}}(\tilde P-\tilde O)$.
Then the $\IZ/4$ action lifts to an action on $\tilde M$, which we also denote by $\rho$, and the following hold:
\begin{enumerate}
 \item $\tilde P(0)=(1+i)/2$, and $\tilde M|_{\tilde S_0}$ is a
 non-trivial two-torsion line bundle.

 \item The action $\tau$ from \eqref{eq: Z action}  lifts to a $\rho$-equivariant action
 $\tilde\tau$ on $\tilde M^\times$.
The quotient $\tilde X_{\Delta_0}=\tilde M^\times/\langle\tilde\tau\rangle$
 is a smooth family of complex two-tori over $\tilde\Delta_0$, and
 $X_{\Delta_0}$ is a resolution of $\tilde X_{\Delta_0}/\langle\rho\rangle$.

 \item On $\tilde M^\times|_{\tilde S_0}$, the action $\tilde\tau$
 is multiplication by a scalar $\alpha$ with $0<|\alpha|<1$.
\end{enumerate}
\end{lem}

\begin{proof}
The action lifts, because $\tilde P$ and $\tilde O$ are invariant under $\rho$.

The two multiplicity-one components of the $\III^*$-fibre are the
outer ends of the two $A_3$-chains in the resolution.  They correspond
to the fixed points $0$ and $(1+i)/2$ of $\rho$.  Since $O$ and $P$
meet different components by Proposition~\ref{prop: surface facts},
their transforms specialize to these two distinct points.  Thus
$\tilde P(0)=(1+i)/2$, proving (i).

The lifting and quotient arguments of
\cite[Proposition~3.1 and Corollary~3.3]{Engel2026} apply with $2P$
in place of $6P$, giving (ii).  Since $2\tilde P(0)=0$, the lifted
action is trivial 
 along
$\tilde S_0$ and acts by a constant scalar $\alpha$ 
 along the fibre direction of
the line bundle. By our choice of $u$, multiplication with $\alpha$ is a contraction, so its norm is at most one. 
\end{proof}

\begin{cor}\label{cor: pullback to product central fibre at 0}
There is a $\lambda\in\IC$ with $\operatorname{Im}\lambda>0$ such that
$ \tilde X_0\isom \hat X_0 /\langle\iota\rangle$ where 
\[
 \hat  X_0=
 \left(
 \frac{\IC}{2\IZ+(1+i)\IZ}
 \times
 \frac{\IC}{\IZ+\lambda\IZ}
 \right) \text { and }
 \iota(c,w)=\left(c+1,w+\frac12\right).
\]
In particular, $\tilde X_0$ is an abelian surface.
\end{cor}

\begin{proof}
The line bundle $\tilde M|_{\tilde S_0}$ is represented by the
character
(\cite[Section~2.2]{birkenhake2004complex})
\[
\varphi\colon \IZ+i\IZ\rightarrow \IC^*, \quad m+in\mapsto (-1)^{m+n}
\]
which is $-1$ on both $1$ and $i$.  Its kernel is
$2\IZ+(1+i)\IZ$, so the double cover
\[
 \hat S_0=\frac{\IC}{2\IZ+(1+i)\IZ}
 \longrightarrow\tilde S_0
\]
trivializes it.  On $\hat S_0\times\IC^*$ the deck involution is
$(c,v)\mapsto(c+1,-v)$, while $\tilde\tau$ acts by
$(c,v)\mapsto(c,\alpha v)$.  Choose $\lambda$ with
$\alpha=\exp(2\pi i\lambda)$.  Passing to the quotient by
$\tilde\tau$ gives the stated description, as in
\cite[Corollary~3.2]{Engel2026}.
\end{proof}

Choose a path from the base point $1/2$ to $\Delta_0^*$ and use
parallel transport to identify
$\Lambda=H_1(X_{1/2},\IZ)$ with $H_1(\tilde X_0,\IZ)$.
Let $T_0$ be the local monodromy around $0$.  The covering
$\tilde M_0^\times\to\tilde X_0$ gives an exact sequence
\[
 0\longrightarrow H_1(\tilde M_0^\times,\IZ)
 \longrightarrow\Lambda\xrightarrow{\xi}\IZ\longrightarrow0,
\]
where $\xi$ records the generator $\tilde\tau$.  It is invariant
under the deck action
of $\tilde X_{\Delta_0}\dashrightarrow X_{\Delta_0}$.

We choose
\begin{equation}\label{eq: preliminary a}
 a\in\frac14\Lambda^{T_0},
 \qquad
 \xi(a)\in\frac14\IZ\setminus\frac12\IZ.
\end{equation}
The class $\bar a$ defines a $\rho$-invariant four-torsion section of
$\tilde X_{\Delta_0}\to\tilde\Delta_0$.  We will make this choice explicit in
Definition~\ref{def: choice of a}.  Following
\cite[Construction~3.5]{Engel2026}, we twist the deck action by
\begin{equation}
 \label{eq: twisted local action at 0}
 \rho'(x)=\rho(x)+\bar a 
\end{equation}
and put $Y_{\Delta_0}=\tilde X_{\Delta_0}/\langle\rho'\rangle$.  Since
$\bar a$ is $\rho$-invariant, we have  \(\rho\bigl(\bar a(s)\bigr)=\bar a(is)\)
and $4\bar a=0$, and
$\rho'$ is indeed
an action of
$\IZ/4$
on $\tilde X_{\Delta_0}$.
The two quotients fit into the diagram
\begin{equation}\label{diag: log transform local}
\begin{tikzcd}
 &&\tilde X_{\Delta_0}\arrow[dll,"/\langle\rho'\rangle"']
          \arrow[drr,"/\langle\rho\rangle"]
          \arrow[d, dashed]&\\
 Y_{\Delta_0}\arrow[drr]&&
 X_{\Delta_0}\arrow[rr,"\text{resolution}"]
 \arrow[d]
 & &
 \tilde X_{\Delta_0}/\langle\rho\rangle
\arrow[dll]\\
& &\Delta_0&&
\end{tikzcd}.
\end{equation}

\begin{prop}\label{prop: log transform at 0}
In the above situation:
\begin{enumerate}
 \item The quotient $Y_{\Delta_0}$ is a complex manifold and proper over $\Delta_0$.
 The lift $a$ determines a biholomorphism
 $Y_{\Delta_0^*}\isom X_{\Delta_0^*}$.

 \item The central fibre
 $Y_0$
 has multiplicity four, and its reduction
 is a bielliptic surface of Bagnera--de Franchis type c\,2) in
 \cite[V.5]{BHPV2004}, with group $\IZ/4\times\IZ/2$.
\end{enumerate}
\end{prop}

\begin{proof}
The action
$\rho'$
is free on the central fibre $\tilde X_0$
because of the translation component $\bar a$,
and it is free on $\tilde X_{\Delta_0}\setminus \tilde X_0$
since $0$ is the only fixed point of the action on the disc. 

For the identification over the punctured disk, we follow
\cite[Lemma~3.6]{Engel2026}: regard $4a$ as a holomorphic period
section on the covering disk, and put
\[
 \sigma_a(s)=\frac{\log s}{2\pi i}\,4a(s)
 \quad\text{in }\tilde X_s,\qquad s\ne0.
\]
Changing the branch of $\log s$ adds an integral multiple of
$4a(s)$, so this is a single-valued holomorphic section.
Since $a$ is invariant under the deck action,
\[
 \sigma_a(is)=\rho(\sigma_a(s))+\bar a(is).
\]
Thus translation $\Psi_a(x,s)=(x+\sigma_a(s),s)$ satisfies
$\Psi_a\circ\rho=\rho'\circ\Psi_a$.
Its inverse induces the required biholomorphism
\[
 Y_{\Delta_0^*}\longrightarrow X_{\Delta_0^*},\qquad
 [(x,s)]_{\rho'}\longmapsto[(x-\sigma_a(s),s)]_\rho.
\]
Here the untwisted quotient is identified with $X_{\Delta_0^*}$,
since the resolution is an isomorphism away from the central fibre.

The map $\tilde\Delta_0\to\Delta_0$ has ramification index four at the
origin, whereas the quotient map $\tilde X_{\Delta_0}\to Y_{\Delta_0}$ is
unramified.  Hence the central fibre has multiplicity four.
Its reduction is the quotient of the abelian surface in
Corollary~\ref{cor: pullback to product central fibre at 0}
by the twisted action. 
This is the $\IZ/2\times\IZ/4$ case, type c\,2), in the
Bagnera--de Franchis classification \cite[V.5]{BHPV2004}, proving (ii).
\end{proof}

We call $Y_{\Delta_0}$ the logarithmic transform of $X_{\Delta_0}$
with parameter $a$.

\begin{defin}
The threefold $Y=Y(u)$ is obtained by gluing $Y_{\Delta_0}$ to
$X(u)|_{\IP^1\setminus\{0\}}$ using the biholomorphism determined by
$a$ in Proposition~\ref{prop: log transform at 0} (i), with $a$ as in
Definition~\ref{def: choice of a}.  The maps to the base give a proper
holomorphic map $f\colon Y\to\IP^1$.
\end{defin}

Thus $f\colon Y\to \IP^1$ is a fibration of a smooth compact complex threefold.  Its fibres away from $0,1,\infty$ are complex two-tori, its fibre at $0$ is the
multiple bielliptic surface above, and its fibres at $1$ and
$\infty$ are unchanged from $\eta \colon X \to \IP^1$.

\section{Local preparations}
\label{Some notation}
We choose $\frac 12$ as a convenient base point and let 
\[ \Lambda = H_1(Y_{\frac12}, \IZ), \quad \Lambda^\times  = H_1(M^\times _{\frac12}, \IZ), \quad \Lambda' = H_1(S_{\frac12}, \IZ).\]
The diagram in Proposition \ref{prop: construction X} shows that there is a filtration on $\Lambda$, which we express through two exact sequences, 
\begin{gather}
	\begin{tikzcd}[ampersand replacement=\&]\label{ses: filtration1}
	0 \rar\& \Lambda^\times \rar \& \Lambda \rar{\xi} \& \IZ \rar \& 0 ,
\end{tikzcd}
\\
	\begin{tikzcd}[ampersand replacement=\&]\label{ses: filtration2}
	0 \rar \& \IZ { \delta} \rar\& \Lambda^\times \rar \& \Lambda' \rar  \& 0 ,
\end{tikzcd}
\end{gather}
compatible with the monodromy action of $\pi_1(\IP^1\setminus \{0, 1, \infty\},\frac12)$,
where $\delta$ denotes the homology class represented by the positively oriented $S^1$ in the $\IC^*$-fibre of $M^\times_{\frac{1}{2}}\rightarrow S_{\frac{1}{2}}$.
We use the same notation for its parallel transports, since monodromy fixes the oriented fibre circle.

Denoting the dual lattice by $V = \Hom(\Lambda, \IZ)$, and similarly defining $V^\times$ and $V'$, the exact sequences become
\begin{gather}
	\begin{tikzcd}[ampersand replacement=\&]
		0 \rar\& \IZ\xi \rar \& V\rar \& V^\times \rar \& 0 ,
	\end{tikzcd}
	\\
	\begin{tikzcd}[ampersand replacement=\&]
		0 \rar \& V' \rar\& V^\times \rar \& \IZ \rar  \& 0 .
	\end{tikzcd}
\end{gather}

For a sufficiently small disc $\Delta_p$ centred at $p\in\IP^1$, the inclusion $Y_p\into Y_{\Delta_p}$ is a homotopy equivalence by \cite[Proposition~C.11]{PetersSteenbrink2008}. Thus, for a smooth fibre $Y_t$ with $t\in\Delta_p$, restriction induces a specialisation map
\begin{equation}
 \label{eq: specialisation map}
 \Sp_p^* \colon H^*(Y_p, \IZ) \to H^*(Y_t, \IZ)^{T_p} \isom \left(\Wedge ^* V\right)^{T_p}
\end{equation}
with image contained in the local monodromy invariants. Here we choose a path from $\frac12$ to $t$ in $\IP^1\setminus\{0,1,\infty\}$ to identify $H^*(Y_t,\IZ)$ with $\Wedge^*V$, and $T_p$ denotes the corresponding local monodromy.
We use $T_p$ also for the induced actions on $V$ and $\Wedge^kV$.
If we represent the original monodromy action by a matrix $A$, then the action on $\Wedge^k V$ is given by the matrix $\Wedge^k((A^{-1})^{\mathsf T})$.

We will start by choosing explicit bases near the three special fibres, compatible with the filtration, and compute the local monodromy and specialisation maps.

Finally, we will show how the chosen local bases compare  to get the global monodromy representation explicitly.

\subsection{Monodromy and specialisation map at $0$}

As a reference space we use the product abelian surface $\hat X_0$ from  Corollary~\ref{cor: pullback to product central fibre at 0}, whose lattice has basis
\[
 e_1=\mat{2\\0},\qquad e_2=\mat{1+i\\0},\qquad
 e_3'=\mat{0\\1},\qquad e_4=\mat{0\\\lambda}.
\]
Its quotient $\tilde X_0$ is the central fibre of the good reduction model and transport along  the real line identifies the homology of $\tilde X_0$ with the homology of the general fibre of $\tilde X_{\Delta_0} \to \tilde \Delta_0$, which in turn is isomorphic to the general fibre of $X_{\Delta_0}\to \Delta_0$. In total we get the following basis of  $\Lambda$.

\begin{prop}\label{prop: monodromy at 0}
A basis of $\Lambda$ is
\[
 \kb_0=(e_1,e_2,e_3,e_4),\qquad
 e_3=\frac12(e_1+e_3')=\mat{1\\1/2}.
\]
For the counterclockwise loop $\gamma_0(r)=\frac12\exp(2\pi ir)$, $0\leq r\leq1$, the local monodromy is
\[
 [T_0]_{\kb_0}=\mat{-1&-1&-1&0\\2&1&1&0\\0&0&1&0\\0&0&0&1}.
\]
Moreover,
\[
 \Lambda^\times=\langle e_1,e_2,e_3\rangle,\qquad
 \delta=2e_3-e_1,\qquad \xi=e_4^*,
\]
where $\delta$ is the
class of the positively oriented fibre circle of $M^\times$.
In particular, $\Lambda^{T_0}=\IZ\delta\oplus\IZ e_4$.
\end{prop}
\begin{proof}
The involution $\iota$ of  Corollary~\ref{cor: pullback to product central fibre at 0}
adds the lattice point $(1,1/2)$, giving the stated index-two overlattice. 

A counterclockwise loop around \(t=0\) on $\Delta_0$ lifts under \(t=s^4\) to a path from \(s_*\) to \(is_*\), where $s_*$ is a base point on $\tilde\Delta_0^*$. Parallel transport along this path, followed by the endpoint identification \(\rho^{-1}:\tilde X_{is_*}\xrightarrow{\sim}\tilde X_{s_*}\), gives the local monodromy. After identifying homology groups \(H_1(\tilde X_{s_*},\mathbb Z)\)
and \(H_1(\tilde X_{is_*},\mathbb Z)\) with that of the central fibre $\tilde X_0$, the monodromy action is exactly the action induced by \(\rho^{-1}\) on \(H_1(\tilde X_0,\mathbb Z)\).
On the central fibre $\tilde X_0$ the linear part of $\rho^{-1}$ is $(c,w)\mapsto(ic,w)$; see \eqref{eq: rho}. Thus
\[
 T_0e_1=-e_1+2e_2,\qquad T_0e_2=-e_1+e_2,\qquad
 T_0e_3=-e_1+e_2+e_3,\qquad T_0e_4=e_4.
\]
The period $e_4$ records the action of $\tilde\tau$ 
 given by Lemma~\ref{lem: good reduction at 0} (ii),
whereas the fibre circle is $e_3'=2e_3-e_1$. 
 This determines the filtration \eqref{ses: filtration1} and \eqref{ses: filtration2}, as well as the invariant lattice $\Lambda^{T_0}$.
\end{proof}

\begin{defin}\label{def: choice of a}
We choose the defining parameter for the logarithmic transform at $0$ to be $a=\frac14e_4$, where $e_4$ 
 is as in Proposition~\ref{prop: monodromy at 0}.
This satisfies \eqref{eq: preliminary a}, since $T_0e_4=e_4$ and $\xi(a)=1/4$.
\end{defin}
We also put
\begin{equation}\label{eq: named forms}
 \xi=e_4^*,\qquad \alpha=e_1^*+e_3^*,\qquad \beta=2e_{12}^*-e_{23}^*,\qquad
 \omega=\beta-e_{34}^*,
\end{equation}
where we use notations specified in Section~\ref{sec: notation}.
These are classes on the reference fibre. The classes $\xi$ and
$\beta$ simplify the local calculation below; $\omega$ and $\alpha$
will be used in the global computation later. In particular, 
\[
 \beta\wedge\xi=2e_{124}^*-e_{234}^*.
\]

\begin{prop}\label{prop: specialisation at 0}
With respect to the dual basis $(e_1^*,e_2^*,e_3^*,e_4^*)$, the local
invariant lattices are
\begin{align*}
 V^{T_0}&=\langle e_3^*,\xi\rangle,&
 (\Wedge^2V)^{T_0}
 &=\langle\omega,e_{34}^*\rangle,&
 (\Wedge^3V)^{T_0}
 &=\langle e_{123}^*,\omega\wedge\xi\rangle.
\end{align*}
The specialisation map 
$
\Sp_0^k\colon H^k(Y_0,\IZ)\longrightarrow(\Wedge^kV)^{T_0}
$
is injective with cokernels
\[
 Q_0^0=0,\qquad Q_0^1\isom\IZ/4,\qquad
 Q_0^2\isom\IZ/2,\qquad Q_0^3\isom\IZ/2,\qquad Q_0^4\isom\IZ/4,
\]
generated in positive degrees by the classes of
\[
 \xi,\qquad \omega,\qquad \omega\wedge\xi,\qquad e_{1234}^*,
\]
respectively.
\end{prop}
\begin{proof}
The invariant lattices follow by taking exterior powers of the
inverse transpose of the homology matrix in
Proposition~\ref{prop: monodromy at 0}, 
 namely, 
\[
 \left([T_0\inverse]_{\kb_0}\right)^{\mathsf{T}}=\mat{1&-2&0&0\\1&-1&0&0\\0&-1&1&0\\0&0&0&1}.
\]
If we put $F=\tilde X_0$ and $G=(Y_0)_{\mathrm{red}}$, the specialisation map is the pullback in cohomology under the degree-four covering
$p\colon F\to G$.

We claim that its kernel is exactly the torsion subgroup of $H^k(G,\IZ)$: composition of pullback and transfer (pushforward) satisfies $ \operatorname{tr}\circ p^*=4\,\mathrm{id}$, so every class in the kernel is $4$-torsion. Conversely, every torsion class maps to zero, because $H^*(F, \IZ)$ is torsion free.  The surface $G$ is of
type~4 in Serrano's classification, with group
$\IZ/4\times\IZ/2$, and $H_1(G,\IZ)$ is torsion free by
\cite[Remark~1.6]{Serrano1990}. The universal coefficient theorem
and Poincar\'e duality then give torsion-free integral cohomology in
every degree. Hence $\Sp_0^k=p^*$ is injective.

We use the product model $F=\hat X_0/\langle\iota\rangle$ from
Corollary~\ref{cor: pullback to product central fibre at 0} to describe the image of $\Sp_0^k$ explicitly. The translation by $a=e_4/4$ lifts to translation by $(0,\lambda/4)$
on $\hat X_0$. The lifts of the $\IZ/4$-actions to $\hat X_0$ are
\[
 \hat\rho(c,w)=(-ic,w+1/4),\qquad
 \hat\rho'(c,w)=(-ic,w+(1+\lambda)/4).
\]
These commute with $\iota$, since $1+i$ is a period of the first
factor. On the universal cover $\IC^2$, the same formula for
$\hat\rho'$ defines an affine map $g$ with
\[
 g^4=t_{(0,1+\lambda)}=t_{2e_3-e_1+e_4}.
\]
The group $\pi_1(G)$ is generated by the lattice translations and
$g$, with this relation and $g t_v g^{-1}=t_{T_0^{-1}v}$ for a translation $v\in \Lambda$.
An invariant class in $\Hom(\Lambda,\IZ)$ descends precisely
when it extends to a homomorphism $\pi_1(G)\to\IZ$, equivalently
when its value on $2e_3-e_1+e_4$ is divisible by four. Thus
\[
 \operatorname{im}\Sp_0^1
 =\{a_3e_3^*+a_4\xi:a_3,a_4\in\IZ,\ 2a_3+a_4\in4\IZ\}.
\]
Thus $Q_0^1\isom\IZ/4$, generated by $\xi$, because 
$e_3^*-2\xi$ is in the image.

In the displayed degree-two basis, $\omega^2=-4e_{1234}^*$,
$\omega\wedge e_{34}^*=2e_{1234}^*$, and $(e_{34}^*)^2=0$.
Thus the intersection matrix is
\[
 \mat{-4&2\\2&0}.
\]
The intersection pairing on $H^2(G,\IZ)$ modulo torsion is
unimodular, and pullback multiplies it by four. Hence the image
has index two. The class $\beta=\omega+e_{34}^*$
descends: it is the pullback of the reduced fibre over the image
of $0$ under
\[
 G\longrightarrow\tilde S_0/\langle-i\rangle\isom\IP^1.
\]
Since $\beta$ and $e_{34}^*$ also form an integral basis, the
cokernel is generated by $[e_{34}^*]=-[\omega]$.

The pairing between the displayed degree-one and degree-three
bases has matrix
\[
 \mat{0&2\\-1&0}.
\]
If the degree-three image has index $q$, unimodularity of the
pairing on $G$ gives $4\cdot q\cdot2=4^2$, hence $q=2$.
The second degree-three generator does not descend, since
\[
 \int_F(e_3^*-2\xi)\wedge\omega\wedge\xi=2
\]
is not divisible by four. Finally, pullback is an isomorphism
in degree zero and multiplication by four in top degree.
\end{proof}

\subsection{Monodromy and specialisation map at $1$}
 Over a neighborhood of $1\in \IP^1$ we can identify the topology of $Y$ directly without use of a good reduction model.
\begin{lem}\label{lem: product at 1}
Over a sufficiently small disc $\Delta_1$ centred at $1$, there is
a homeomorphism
\[
 Y_{\Delta_1}\cong S_{\Delta_1}\times S^1\times S^1
\]
over $\Delta_1$. Under this identification, $\delta$ is the
positively oriented first circle, and $\xi$ is induced by
projection to the second circle.
\end{lem}
\begin{proof}
Work first over the closed disk $\overline\Delta_1$.
The surface $S_{\overline\Delta_1}$ retracts onto the cuspidal fibre $S_1$, which is
homeomorphic to $S^2$. Since $M|_{S_1}$ has degree zero, $M|_{S_{\overline\Delta_1}}$
is differentiably trivial. In a smooth trivialisation the action
\eqref{eq: Z action} is
\[
 (s,m)\longmapsto(t_{-2P}(s),\lambda(s)m).
\]
For $u$ sufficiently small, we have  $|\lambda|<1$ on $S_{\overline\Delta_1}$ and, 
as $S_{\overline\Delta_1}$ is simply connected, we may write $\lambda=\exp L$ with
$L$ smooth and $\operatorname{Re}L<0$.

The smooth locus of $S_1$ is the additive group $\IC$, so its
exponential map is an isomorphism. The relative exponential map
therefore gives, after shrinking $\Delta_1$, a holomorphic logarithm $v$
of the section $-2P$. Translation by $\exp(rv)$, $0\leq r\leq1$,
defines a fibre-preserving smooth isotopy $\varphi_r$ from the
identity to $t_{-2P}$. Here the group action extends to $S_{\overline\Delta_1}$ and
fixes the cusp.

Since $|\lambda|<1$, a fundamental annulus 
\(A=\left\{(s,m):
|\lambda(t_{2P}(s))|\le |m|\le1\right\}\)
in $S_{\Delta_1}\times\mathbb C^*$
identifies the quotient 
\(Y_{\Delta_1}\cong
(S_{\Delta_1}\times\mathbb C^*)/\langle\tau\rangle
\)
with the mapping torus of
\[
 (s,z)\longmapsto(t_{-2P}(s),e^{i\operatorname{Im}L(s)}z)
 \quad\text{on }S_{\Delta_1}\times S^1.
\]
Its gluing map is isotopic to the identity over $\Delta_1$ via
$(s,z)\mapsto(\varphi_r(s),e^{ir\operatorname{Im}L(s)}z)$,
giving the required product. The angular circle 
 $|z|=1$ with counterclockwise direction
represents $\delta$. 
 The radial interval of $A$ becomes a second circle after boundary gluing, recording one application of \(\tau\),
so projection to it induces $\xi$.
\end{proof}

\begin{prop}\label{prop: monodromy and specialisation at 1}
There is an integral basis $\kb_1=(f_1,f_2,f_3,f_4)$ of $\Lambda$,
compatible with the filtration 
\eqref{ses: filtration1} and \eqref{ses: filtration2}:
\[
 \qquad \delta=f_3,\qquad 
\Lambda^\times=\langle f_1,f_2,f_3\rangle,\qquad \xi=f_4^*.
\]
The monodromy of the counterclockwise loop
$\gamma_1(r)=1-\frac12\exp(2\pi ir)$, $0\leq r\leq1$ is
\[
 [T_1]_{\kb_1}=\mat{1&1&0&0\\-1&0&0&0\\0&0&1&0\\0&0&0&1}.
\]
The local invariant lattices are
\begin{align*}
 V^{T_1}&=\langle f_3^*,f_4^*\rangle, &&
 (\Wedge^2V)^{T_1}&=\langle f_{12}^*,f_{34}^*\rangle, &&
 (\Wedge^3V)^{T_1}&=\langle f_{123}^*,f_{124}^*\rangle.
\end{align*}
In degrees zero and four the whole lattices 
$\Wedge^0 V$ and $\Wedge^4 V$
are invariant.
For every $k$, specialisation induces an isomorphism
$\Sp_1^k\colon H^k(Y_1,\IZ)\isom(\Wedge^kV)^{T_1}$.
\end{prop}
\begin{proof}
By Lemma~\ref{lem: product at 1}, the two circle factors give
monodromy-invariant classes $f_3=\delta$ and $f_4$ with
$\xi(f_4)=1$. Together with a basis $f_1,f_2$ 
of $\Lambda'=H_1(S_{\frac{1}{2}},\IZ)$
for the type $\DI$
monodromy on the elliptic factor as in \cite[V.10]{BHPV2004},
these give the asserted basis
 of $\Lambda$ and  matrix  $[T_1]_{\kb_1}$.
 The homeomorphism  $Y_{\Delta_1}\cong S_{\Delta_1}\times (S^1)^2$ of Lemma~\ref{lem: product at 1} identifies $Y_t$ with $S_t\times (S^1)^2$ for each $t\in \Delta_1$.
Since $S_1$ is homeomorphic to $S^2$, the specialisation map for
$S_{\Delta_1}\to\Delta_1$ is an isomorphism in degrees zero and two. In degree
two, 
the isomorphism
follows because the cohomology class of the section $O$ has
degree one 
and hence generates the second cohomology groups
on both 
 central and nearby
fibres
 of $S_{\Delta_1}\to\Delta_1$.

The displayed invariant lattices
and the specialisation assertion now follow from the K\"unneth formula for the product $S_t\times (S^1)^2$. More precisely, put \(C=(S^1)^2\), with
\[
H^1(C,\mathbb Z)=\langle f_3^*,f_4^*\rangle.
\]
Then for $k\geq 0$,
\[
H^k(Y_1,\mathbb Z)
\cong
\bigoplus_{p+q=k}H^p(S_1,\mathbb Z)\otimes H^q(C,\mathbb Z). 
\]
Since \(S_1\) is homeomorphic to \(S^2\), only \(p=0,2\) contribute in the above direct sum.
On the nearby elliptic curve $S_t$, monodromy has no invariants in \(H^1(S_t,\mathbb Z)\), whereas it acts identically on \(H^0\) and \(H^2\). It also acts identically on \(H^*(C,\mathbb Z)\). Consequently,
\[
H^k(Y_t,\mathbb Z)^{T_1}
\cong
\bigoplus_{\substack{p+q=k\\p=0,2}}
H^p(S_t,\mathbb Z)\otimes H^q(C,\mathbb Z).
\]These are exactly the summands appearing for the central fibre, and specialization is an integral isomorphism on each.
\end{proof}

\subsection{Monodromy and specialisation map at $\infty$}

\begin{prop}\label{prop: monodromy and specialisation at infinity}
There is an integral basis $\kb_\infty=(h_1,h_2,h_3,h_4)$ of
$\Lambda$, compatible with the filtration \eqref{ses: filtration1} and \eqref{ses: filtration2}, such that
\[
 \Lambda^\times=\langle h_1,h_2,h_3\rangle,\qquad
 \delta=h_3,\qquad \xi=h_4^*.
\]
Here $\delta$ is the positive fibre circle, and
\[
 [T_\infty]_{\kb_\infty}
 =\mat{1&1&0&0\\0&1&0&0\\0&0&1&1\\0&0&0&1}.
\]
In particular, for $N_\infty=T_\infty-I$,
\[
 N_\infty^2=0,\qquad
 \operatorname{im}N_\infty=\ker N_\infty
 =\IZ h_1\oplus\IZ h_3.
\]
For every $k$, specialisation induces an isomorphism
\[
 \Sp_\infty^k\colon H^k(Y_\infty,\IZ)
 \isom(\Wedge^kV)^{T_\infty}.
\]
\end{prop}
\begin{proof}
Use the multiplicative uniformisation
 \(Y_q=X_q\cong
(\mathbb C^*)^2/\langle p_1(q),p_2(q)\rangle\)
of
Proposition~\ref{prop: compactification X}, with periods
$p_1(q)=(c_1(q)q,c_2(q))$ and $p_2(q)=(c_3(q),c_4(q)q)$, 
 where the functions $c_j$ are holomorphic and nowhere zero on the disc.
The positive coordinate circles are
$h_1,h_3$, and logarithms of the two periods give $h_2,h_4$.
More precisely, use the universal covering
\[
E:\mathbb C^2\to(\mathbb C^*)^2,\qquad
E(u,v)=(e^{2\pi iu},e^{2\pi iv}).
\]
Its kernel is generated by
\[
h_1=(1,0),\qquad h_3=(0,1).
\]
Choose logarithms at a reference value \(q_*\). The two extra lattice generators are
\[
h_2=\frac1{2\pi i}
\bigl(\log(c_1(q_*)q_*),\log c_2(q_*)\bigr),
\]\[
h_4=\frac1{2\pi i}
\bigl(\log c_3(q_*),\log(c_4(q_*)q_*)\bigr).
\]Thus
\[
Y_{q_*}\cong
\mathbb C^2/
\bigl(\mathbb Zh_1+\mathbb Zh_2+\mathbb Zh_3+\mathbb Zh_4\bigr).
\]
The first three already describe \(M^\times_{q_*}\); the fourth records one application of \(\tau\). Hence
\[
\Lambda^\times=\langle h_1,h_2,h_3\rangle,\qquad
\delta=h_3,\qquad \xi(h_4)=1.
\]
Continuation around the loop 
\[
q(r)=q_*e^{2\pi ir}, \quad0\leq r\leq 1
\]
changes \(\log q\) by \(2\pi i\). Each \(c_j\), however, has a single-valued logarithm on the whole disc, so its logarithm acquires no increment. Therefore
\[
h_2\longmapsto h_2+(1,0)=h_2+h_1,
\qquad
h_4\longmapsto h_4+(0,1)=h_4+h_3.
\]The coordinate circles $h_1$ and $h_3$ remain unchanged. Consequently,
\[
T_\infty h_1=h_1,\quad
T_\infty h_2=h_1+h_2,\quad
T_\infty h_3=h_3,\quad
T_\infty h_4=h_3+h_4,
\]
which gives exactly the displayed matrix in the stated order. Finally, these formulas immediately give
\[
N_\infty h_2=h_1,\qquad N_\infty h_4=h_3,\qquad
N_\infty h_1=N_\infty h_3=0,
\]so \(N_\infty^2=0\) and
\[
\operatorname{im}N_\infty=\ker N_\infty
=\mathbb Zh_1\oplus\mathbb Zh_3.
\]

The filling is the same Mumford construction as in Engel, so the
specialisation assertion in degrees one to three is exactly
\cite[Lemma~7.4]{Engel2026}. Degree zero follows from connectedness.
In degree four, the central fibre is reduced and irreducible, and
the nearby fundamental class specialises to its fundamental class
with coefficient one. Hence specialisation is an isomorphism in
 degree $4$ as well.
\end{proof}

\subsection{Global monodromy over $U = \IP^1\setminus\{0,1,\infty\}$}
We now compare the markings 
 at $0$ and $1$; the monodromy at infinity then follows from the relation $T_\infty T_1T_0=I$, coming from our chosen presentation of the fundamental group of $U$.

\begin{prop}\label{prop: global monodromy}
The local markings may be chosen so that
\[
 e_1=2f_1+f_3,\qquad e_2=f_2,\qquad
 e_3=f_1+f_3,\qquad e_4=f_2+f_4.
\]
In particular,
$\delta=2e_3-e_1=f_3$.
We use $\kb_0=(e_1,e_2,e_3,e_4)$, the marking at $0$, as the
common reference basis. In this basis the monodromies are
\begin{align*}
 [T_0]_{\kb_0}&=\mat{-1&-1&-1&0\\2&1&1&0\\0&0&1&0\\0&0&0&1},&
 [T_1]_{\kb_0}&=\mat{1&1&0&1\\-2&0&-1&-1\\0&-1&1&-1\\0&0&0&1},\\
 [T_\infty]_{\kb_0}&=\mat{1&0&0&-1\\-2&0&-1&1\\2&1&2&1\\0&0&0&1}.
\end{align*}
In particular, the parameter of the logarithmic transform is
$a=\frac14e_4=\frac14(f_2+f_4)$.
\end{prop}
\begin{proof}
As in \cite[Theorem~5.1]{Engel2026}, first compare the quotient lattice
$\Lambda'=\Lambda^\times/\IZ\delta$, writing $[l]$ for the class of
$l\in\Lambda^\times$. Denote by $T_p'$ the operator
induced by $T_p$ on $\Lambda'$, for $p=0,1,\infty$.
Since $\delta=2e_3-e_1$ and thus $[e_1]=2[e_3]$, the pair
$([e_3],[e_2])$ is an integral oriented basis.  The local computations give
\[
 A:=[T_0']_{([e_3],[e_2])}=\mat{-1&-2\\1&1},\qquad
 B:=[T_1']_{([f_1],[f_2])}=\mat{1&1\\-1&0}.
\]
Since $T_0'$ has trace zero and determinant one, write
\[
 [T_0']_{([f_1],[f_2])}=\mat{r&s\\t&-r},\qquad r^2+st=-1.
\]
By Proposition~\ref{prop: monodromy and specialisation at infinity},
$T_\infty'$ has trace two and determinant one. Thus
\[
 \operatorname{tr}(T_1'T_0')
 =\operatorname{tr}((T_\infty')^{-1})=2,
\]
which, multiplying the two matrices in the $([f_1],[f_2])$ basis,
gives $r+t-s=2$. Eliminating $t$ gives
$(2s+2-r)^2=-r(3r+4)$, so $r=-1$ or $0$.
The three possibilities for $[T_0']_{([f_1],[f_2])}$ are
\[
 A,\qquad
 BAB^{-1}=\mat{-1&-1\\2&1},\qquad
 B^2AB^{-2}=\mat{0&-1\\1&0}.
\]
If the matrix is $B^jAB^{-j}$, replace the basis by
$((T_1')^j[f_1],(T_1')^j[f_2])$. The matrix of $T_0'$ then becomes
$A$, while that of $T_1'$ remains $B$.

Both bases now give the same matrix $A$ for $T_0'$, but they need
not yet agree: they may differ by a base change matrix $C\in\mathrm{SL}_2(\IZ)$ such that $\inverse C A C = A$. Solving this for $C$ gives 
\[
 C=\mat{x&-2y\\y&x+2y},\qquad
 1=\det C=(x+y)^2+y^2,
\]
and hence $C\in\{I,A,-I,-A\}
{ =\{I,A,A^2,A^3\}}$. 
{Thus we may write $C=A^k$ for some $0\leq k\leq 3$.}
Replacing the marking at $0$ by
$((T_0')^k[e_3],(T_0')^k[e_2])$ therefore makes the two bases agree:
$[e_3]=[f_1]$ and $[e_2]=[f_2]$.
These changes of quotient markings are induced on the full lattice
$\Lambda$ by $T_1^j$ and $T_0^k$, respectively.
They preserve the local monodromy matrices,
$\delta$ and $\xi$; the change at $0$ also fixes $e_4$, hence the
chosen logarithmic-transform parameter from Definition \ref{def: choice of a}.

We now compare the lifts to the full lattice $\Lambda$.
Every identification inducing the chosen quotient marking and
respecting $\delta=2e_3-e_1=f_3$ and $\xi$ now has the form
\[
 e_1=2f_1+(2b-1)f_3,\quad e_2=f_2+r f_3,\quad
 e_3=f_1+b f_3,\quad e_4=f_4+p f_1+q f_2+h f_3
\]
for integers $b,r,p,q,h$. Put $k=2b-1$. Computing
$N_\infty=
 T_\infty-I = 
(T_1T_0)^{-1}-I$ in this marking
 $\{f_1,f_2,f_3,f_4\}$
gives
\[
 N_\infty^2 f_1=-r f_3,\qquad
 N_\infty^2 f_4\equiv p(f_1-f_2)\pmod{\IZ f_3}.
\]
Thus $r=p=0$ by Proposition~\ref{prop: monodromy and specialisation at infinity}.
With these values,
\[
 N_\infty f_1=N_\infty f_2=v:=f_1-f_2+k f_3,\qquad
 N_\infty f_3=0,\qquad N_\infty f_4=-2qv+kq f_3.
\]
We have thus 
\[
N_\infty \Lambda^\times = N_\infty \langle f_1,f_2,f_3\rangle =  \IZ v.
\]
Since $\xi(f_4)=\xi(h_4)=1$, one has $f_4-h_4\in \ker \xi=\Lambda^\times$ and hence 
\begin{equation}\label{eq: N_infty f4}
 N_\infty f_4\equiv N_\infty h_4=\delta \quad \mathrm{modulo}\,  N_\infty \Lambda^\times,   
\end{equation}
where the latter equality is by Proposition~\ref{prop: monodromy and specialisation at infinity}. On the other hand, since $f_3=\delta$,
\begin{equation}\label{eq: N_infty f4'}
N_\infty f_4=-2qv+kq f_3 \equiv kq \delta\quad  \mathrm{modulo}\,  N_\infty \Lambda^\times
\end{equation}
Combining \ref{eq: N_infty f4} and  \ref{eq: N_infty f4'} gives $kq=1$.

Simultaneously 
changing the markings $\{e_1,e_2,e_3,e_4\}$ and $\{f_1,f_2,f_3,f_4\}$ to
\[\begin{aligned}
e_1'&=-e_1,& e_2'&=-e_2,&
e_3'&=\delta-e_3,& e_4'&=e_4,\\
f_1'&=-f_1,& f_2'&=-f_2,&
f_3'&=f_3,& f_4'&=f_4,
\end{aligned}\]
replaces
$(k,q)$ by $(-k,-q)$ without changing the local matrices  of $T_0$ at $0$ and $T_1$ at $1$, $\delta$,
or $e_4$. 
Thus
we may take $k=q=1$, hence $b=1$.
Finally, replacing $f_4$ by $f_4+h f_3$ removes $h$.
This gives the stated relations between the bases. The matrix of
$T_0$ in $\kb_0$ is already given in
Proposition~\ref{prop: monodromy at 0}. Applying $T_1$ to these
relations and expressing the results in the $e_i$ gives the second
matrix. Finally, $T_\infty=(T_1T_0)^{-1}$ gives the last one.
\end{proof}

\begin{cor}\label{cor: local invariants common basis}
With the notation of \eqref{eq: named forms}, the local invariant lattices at $0$ and
$1$ and the full monodromy invariant lattices are
\[
\begin{array}{@{}cccc@{}}
\toprule
 k&(\Wedge^kV)^{T_0}&(\Wedge^kV)^{T_1}&
 (\Wedge^kV)^{\langle T_0,T_1\rangle}\\
\midrule
 0&\IZ&\IZ&\IZ\\
 1&\langle e_3^*,\xi\rangle&\langle\alpha,\xi\rangle&\IZ\xi\\
 2&\langle\omega,e_{34}^*\rangle&\langle\omega,\alpha\wedge\xi\rangle&\IZ\omega\\
 3&\langle e_{123}^*,\omega\wedge\xi\rangle&
 \langle\omega\wedge\alpha,\omega\wedge\xi\rangle&\IZ(\omega\wedge\xi)\\
 4&\IZ(\omega\wedge\alpha\wedge\xi)&\IZ(\omega\wedge\alpha\wedge\xi)&\IZ(\omega\wedge\alpha\wedge\xi)\\
\bottomrule
\end{array}
\]
Here $\omega\wedge\alpha\wedge\xi=e_{1234}^*$ is the positive
generator in top degree.
\end{cor}
\begin{proof}
The column at $0$ is Proposition~\ref{prop: specialisation at 0}.
Dualising the basis relations in Proposition~\ref{prop: global monodromy}
and taking wedge products gives
\begin{align*}
 f_3^*&=\alpha,& f_4^*&=\xi,\\
 f_{12}^*&=\omega+2\alpha\wedge\xi,& f_{34}^*&=\alpha\wedge\xi,\\
 f_{123}^*&=\omega\wedge\alpha,& f_{124}^*&=\omega\wedge\xi.
\end{align*}
Subtracting twice the second degree-two generator from the first
gives the stated integral basis at $1$.

Since the meridians at $0$ and $1$ generate the fundamental group
of $\IP^1\setminus\{0,1,\infty\}$, the full invariant lattices
are the intersections of the two local invariant columns. These
are read off from the table, using $\alpha=e_1^*+e_3^*$,
$\alpha\wedge\xi=e_{14}^*+e_{34}^*$ and
$\omega\wedge\alpha=e_{123}^*-e_{134}^*$.
\end{proof}

\section{Computing the topological invariants of $Y$}

The topological invariants of $Y$ are computed with the standard theorems of algebraic topology, that is, the Seifert--van Kampen theorem and the Leray spectral sequence for the locally constant sheaf $\IZ$ on $Y$. 
The actual computations require some care, building upon the local results from the previous section.

\subsection{Computation of the fundamental group}
We apply the Seifert--van Kampen theorem as in
\cite[Section~6]{Engel2026}. Put $U=\IP^1\setminus\{0,1,\infty\}$.
Since $U$ retracts onto a graph,
{ which has trivial $\pi_2$,}
and $Y_U\to U$ is a torus bundle
with a section
$\mathfrak{s}\colon U\rightarrow Y_U$,
we have
\[
 \pi_1(Y_U)=\Lambda\rtimes\pi_1(U),
\]
 where the action of $\pi_1(U)$ on $\Lambda$ is the monodromy representation.
Let $x_0,x_1,x_\infty$ be the meridians lifted to this section,
with $x_\infty x_1x_0=1$. The inclusion $Y_U\into Y$ is surjective
on fundamental groups, since loops can be moved off the three
omitted fibres.

\begin{prop}\label{prop: simply connected}
The complex manifold $Y$ is simply connected.
\end{prop}
\begin{proof}
The two multiplicative circles of the Mumford construction vanish
in $Y_{\Delta_\infty}$, exactly as in \cite[Section~6]{Engel2026}.
In the reference basis at $0$, Proposition~\ref{prop: global monodromy}
gives
\[
 K_\infty=(T_\infty-I)\Lambda
 =\IZ(e_1-e_2-e_3)\oplus\IZ(2e_3-e_1).
\]
The normal closure of $K_\infty$ in $\pi_1(Y_U)$
 contains $T_1(e_1-e_2-e_3)=-e_2$,
and hence contains $\ker\xi=\langle e_1,e_2,e_3\rangle$.
Thus the image of the fibre group 
 $\Lambda$ in $\pi_1(Y)$
is generated by the image $c$
of $e_4$. It is central, since $\xi$ is monodromy invariant.

The section
$\mathfrak{s}\colon U\rightarrow Y_U$
extends over the unchanged fibres at $1$ and $\infty$
by Propositions~\ref{prop: construction X} and
\ref{prop: compactification X}. 
Thus we can fill in the loops over $1$ and $\infty$ with a disk and get
$x_1=x_\infty=1
\in \pi_1(Y)
$, so the base relation
 $x_\infty x_1x_0=1$
also gives $x_0=1$.

It remains to use the logarithmic filling at $0$.
The gluing is translation by
\[
 \sigma_a(s)=\frac{\log s}{2\pi i}\,4a
\]
on the fourth-root disk
 $\tilde  \Delta_0$.
Four positive turns of the base meridian
give one turn in $s$, so the lifted translation changes by
$4a=e_4$. 
After killing
$\ker\xi$, the resulting filling relation is therefore
\[
 x_0^4=c^{\xi(4a)}=c,
\]
since $a=\frac14e_4$ and $\xi(e_4)=1$.
As $x_0=1$, this kills $c$ as well. All generators of
$\pi_1(Y_U)$ are thus trivial in $\pi_1(Y)$.
\end{proof}
\subsection{Computation of the integral cohomology}
We will compute the integral cohomology of $Y$ using the 
Leray spectral sequence.
First we will compute the groups on the $E_2$ page and then analyse the differentials, following the same path as \cite{Engel2026}.

\subsubsection{The $E_2$-page of the spectral sequence}
Let $j\colon U\into\IP^1$, and let $\mathcal L_k$ be the local
system on $U$ with fibre $L_k=\Wedge^kV$ induced by the monodromy representation.
\begin{lem}\label{lem: global invariant lattices}
The composition
\[
 H^0(U,\mathcal L_k)\longrightarrow(\Wedge^kV)^{T_0}
 \longrightarrow Q_0^k
\]
is surjective.
\end{lem}
\begin{proof}
The global sections $H^0(U,\mathcal L_k)=(\Wedge^kV)^{\pi_1(U)}$
are computed in Corollary~\ref{cor: local invariants common basis} and by Proposition \ref{prop: specialisation at 0} they map onto  the respective $Q_0^k$. 
\end{proof}

\begin{prop}\label{prop: cohomology of direct images}
In the notation defined in \eqref{eq: named forms}, the $E_2$-page of the Leray spectral sequence
\[
 E_2^{p,q}=H^p(\IP^1,R^q f_*\IZ)\Longrightarrow H^{p+q}(Y,\IZ)
\]
is as follows, with the only possibly non-zero differentials $d_2$ shown:
\[
\begin{tikzcd}[column sep= small,row sep=small]
 q=4 & 4\IZ(\omega\wedge\alpha\wedge\xi)
 \arrow[drr] & 0
 & \IZ[\omega\wedge\alpha\wedge\xi]\\
 q=3 & 2\IZ(\omega\wedge\xi)
 \arrow[drr] & 0
 & \IZ[\omega\wedge\alpha]\\
 q=2 & 2\IZ\omega
 \arrow[drr] & 0
 & \IZ[\alpha\wedge\xi]\\
 q=1 & 4\IZ\xi
 \arrow[drr] & 0 & \IZ[\alpha]\\
 q=0 & \IZ & 0 & \IZ[1]\\
 & p=0 & p=1 & p=2
\end{tikzcd}
\]
The generators in the $p=2$ column are given as classes of coinvariants with respect to $\pi_1(U)$ via the isomorphism 
 $H^2(\IP^1,R^kf_*\IZ)\isom (L_k)_{\pi_1(U)}$ explained in the proof.  In particular, $[1]$ corresponds to the positive generator of $H^2(\IP^1,\IZ)$.
\end{prop}
\begin{proof}
Consider the natural map $R^k f_*\IZ\to j_*\mathcal L_k = j_*j^*R^k f_*\IZ$. Over $U$ this is an isomorphism. For $p\in \{ 0 , 1, \infty\}$ we take sections over a small disc $\Delta_p$ around $p$ getting $H^k(Y_{\Delta_p}, \IZ)\isom H^k(Y_p, \IZ) \to (\Wedge^kV)^{T_p}$, which is exactly given by the specialisation map \eqref{eq: specialisation map}, see  \cite[(5.4)]{Looijenga1992}.

Therefore, by Propositions \ref{prop: specialisation at 0}, \ref{prop: monodromy and specialisation at 1}, and \ref{prop: monodromy and specialisation at infinity}, this map fits into a short exact sequence 
\begin{equation}\label{eq: direct image sheaves}
 0\longrightarrow R^k f_*\IZ\longrightarrow j_*\mathcal L_k
 \longrightarrow(Q_0^k)_0\longrightarrow0,
\end{equation}
with cokernel a skyscraper sheaf supported at $0$. 

By Lemma \ref{lem: global invariant lattices} taking global sections gives a short exact sequence
\[ 0 \to H^0(\IP^1, R^kf_*\IZ) \to
H^0(\IP^1,j_*\mathcal L_k) \to Q_0^k\to 0. \]
 Combining the orders $|Q_0^k| = 1,4,2,2,4$ from Proposition \ref{prop: specialisation at 0} with the descriptions of generators of the middle group from Corollary \ref{cor: local invariants common basis} gives the groups in the first column of the table. 

The long exact cohomology sequence of
\eqref{eq: direct image sheaves} gives isomorphisms
\begin{equation}\label{eq: ident cohomology}
 H^1(\IP^1,R^kf_*\IZ)\isom H^1(\IP^1,j_*\mathcal L_k),
 \qquad
 H^2(\IP^1,R^kf_*\IZ)\isom H^2(\IP^1,j_*\mathcal L_k),
\end{equation}
because the skyscraper quotient has no cohomology in positive
degree.

We next compute the second cohomology, which is the group of
coinvariants by \cite[(5.3)]{Looijenga1992}, that is, 
\[
 H^2(\IP^1,j_*\mathcal L_k)
 \isom (L_k)_{\pi_1(U)}
 = \frac{\Wedge^kV}
 {\im(T_0-I)+\im(T_1-I)}.
\]
In the indexed dual bases, these actions are computed from the
inverse transposes of the matrices in
Proposition~\ref{prop: global monodromy} and their exterior powers.

Recalling the notation from \eqref{eq: named forms}:
\[
\xi=e_4^*,\qquad \alpha=e_1^*+e_3^*,\qquad \beta=2e_{12}^*-e_{23}^*,\qquad
\omega=\beta-e_{34}^* = 2e_{12}^*-e_{23}^*-e_{34}^* ,
\]
integral column operations give the following  complete set of relations 
for $(L_k)_{\pi_1(U)}$:
\[
\begin{array}{@{}cll@{}}
\toprule
 k&\text{relations}&\text{generator}\\
\midrule
 1&e_2^*=\xi=0,\quad e_3^*=-2e_1^*&[\alpha]= -[e_1^*]\\
 2&e_{13}^*=e_{23}^*=e_{24}^*=0,\quad
 e_{14}^*=e_{12}^*,\quad e_{34}^*=-2e_{12}^*&[\alpha\wedge\xi] = -[e_{12}^*]\\
 3&e_{124}^*=e_{134}^*=e_{234}^*=0&[\omega\wedge\alpha] = [e_{123}^*]\\
\bottomrule
\end{array}
\]
Thus each quotient is infinite cyclic, without torsion.
Degrees zero and four are constant. This proves the $H^2$ column.

For later use we record that in $(L_2)_{\pi_1(U)}$ we have the relation
\begin{equation}\label{eq: invariant coinvariant relation}
 [\omega]=-4[\alpha\wedge\xi].
\end{equation}

For the remaining first cohomology group we use the 
description of Dettweiler and Wewers
\cite[\S1.3]{DettweilerWewers2006} of the image of the inclusion $
 \iota\colon H^1(\IP^1, j_* \kl_k)\into H^1(U, \kl_k)$ in terms of group cohomology,
changing their right-action convention to our left-action
convention, with $T_i$ acting on $L_k$.

A group cocycle is a map
\[
 c\colon\pi_1(U)\longrightarrow L_k
\]
satisfying $ c(\gamma\delta)=c(\gamma)+\gamma c(\delta)$ and a coboundary associated with $v\in L_k$ is $d_v(\gamma)=(\gamma-I)v$.

A class $[c]\in H^1(U, \kl_k)$ is in the image
 of $\iota$
if and only if it extends over the punctures, that is, if we restrict the action to the cyclic group generated by one of the loops $\gamma_i$,
 $i\in\{0,1,\infty\}$,
then the class $c(\gamma_i)\in(T_i-I)L_k$ is a coboundary 
\cite[Lemma~1.2]{DettweilerWewers2006}.

So if $[c]$ is in the image 
 of $\iota$,
we can globally modify
 $c$
by  a coboundary
 $d_v$ with $c(\gamma_0)=(T_0-I)v$,
so that
$c(\gamma_0)=0$.
Puting  $x=c(\gamma_1)$ we have 
$x\in(T_1-I)L_k$.
The relation
$\gamma_\infty\gamma_1\gamma_0=1$ and the cocycle rule give
\[
 0=c(1)=c(\gamma_\infty)+T_\infty c(\gamma_1)
       +T_\infty T_1c(\gamma_0)
       =c(\gamma_\infty)+T_\infty x.
\]
Thus $ x = -\inverse T_\infty c(\gamma_\infty)$. 
 The condition that $c(\gamma_\infty)\in (T_\infty-I)L_k$
is then equivalent to $x \in \inverse T_\infty (T_\infty -I)L_k = (T_\infty -I)L_k $.
 Finally, the normalization
  $c(\gamma_0)=0$
 at $0$ is
unique only up to coboundaries $d_v$ with
$v\in L_k^{T_0}$, which change $x$ by $(T_1-I)v$. We obtain
\begin{equation}\label{eq: parabolic cohomology}
 H^1(\IP^1,j_*\mathcal L_k)
 \isom
 \frac{(T_1-I)L_k\cap(T_\infty-I)L_k}
 {(T_1-I)(L_k^{T_0})}.
\end{equation}
Integral column operations
 on $(T_1-I, T_\infty-I)$
give the intersections in the numerator:
\[
\begin{array}{@{}ccc@{}}
\toprule
 k&(T_1-I)L_k\cap(T_\infty-I)L_k
   &\text{
    preimage
   {in $L_k^{T_0}$ under $T_1-I$}
   }\\
\midrule
 1&\IZ(2e_1^*+e_2^*+e_3^*+\xi)&e_3^*\\
 2&\IZ(2e_{14}^*+e_{24}^*+e_{34}^*)&\beta\\
 3&\IZ(e_{124}^*-e_{134}^*-e_{234}^*)&e_{123}^*\\
\bottomrule
\end{array}
\]
The last column lies in $L_k^{T_0}$ (see Proposition \ref{prop: specialisation at 0}), and its image under
$T_1-I$ is the displayed generator
 of $(T_1-I)L_k\cap(T_\infty-I)L_k$ in the second column.
Indeed, the monodromy matrix gives
\begin{align*}
 (T_1-I)e_3^*&=2e_1^*+e_2^*+e_3^*+\xi,\\
 (T_1-I)\beta=(T_1-I)e_{34}^*
 &=2e_{14}^*+e_{24}^*+e_{34}^*,\\
 (T_1-I)e_{123}^*&=e_{124}^*-e_{134}^*-e_{234}^*.
\end{align*}
Hence numerator and denominator
in \eqref{eq: parabolic cohomology} agree and the quotient is zero. 
The same holds in degrees
zero and four, where the local systems are constant.
Thus $ H^1(\IP^1,R^kf_*\IZ)=H^1(\IP^1,j_*\mathcal L_k)=0$ as claimed.
\end{proof}

\subsubsection{Computing $d_2$}
By Proposition~\ref{prop: cohomology of direct images}, the only possible non-zero
differentials in the integral Leray spectral sequence are 
\[
 d_2^{0,q}\colon E_2^{0,q}\longrightarrow E_2^{2,q-1},
 \qquad q=1,\ldots,4
\]
and since all groups involved are infinite cyclic, we only need to compute the image of the generators. We will do this directly for $4\xi$ and then exploit the multiplicative structure of the spectral sequence.

\begin{prop}\label{prop: Leray differentials}
With the transgression sign convention specified below, the
differentials are
\begin{align*}
 d_2(4\xi)&=[1],\\
 d_2(2\omega)&=[\alpha],\\
 d_2(2\omega\wedge\xi)&=-[\alpha\wedge\xi],\\
 d_2(4\omega\wedge\alpha\wedge\xi)&=-[\omega\wedge\alpha].
\end{align*}
\end{prop}
\begin{proof}
The low-degree exact sequence of the Leray spectral sequence contains
\[
 H^1(Y,\IZ)\longrightarrow H^0(\IP^1,R^1f_*\IZ)
 \xrightarrow{d_2}H^2(\IP^1,\IZ).
\]
Thus $d_2$ is the obstruction to extending a section of
$R^1f_*\IZ$ to a class in $H^1(Y,\IZ)$.
Put $B=\IP^1\setminus\{0\}$. Since $B$ is contractible and the
fibres are connected, the same sequence
 for $Y_B\rightarrow B$
gives
\[
 H^1(Y_B,\IZ)\isom H^0(B,R^1f_*\IZ).
\]
Using the restriction of \eqref{eq: direct image sheaves} to $B$, the class $4\xi$,
 being a section of $R^1f_*\IZ$ over $B$,
extends uniquely to $u_B\in H^1(Y_B,\IZ)$.  Let
$u_0\in H^1(Y_{\Delta_0},\IZ)$ be its local extension, supplied
by Proposition~\ref{prop: specialisation at 0}.
Their restrictions to $Y_{\Delta_0^*}$ agree on the fibres.
The low-degree sequence over $\Delta_0^*$ therefore gives a unique
$\nu\in H^1(\Delta_0^*,\IZ)$ with
\[
 u_0-u_B=f^*\nu.
\]
Computing Leray and Mayer--Vietoris with the same injective
resolution identifies the gluing obstruction
 for $u_0$ and $u_B$
with
$d_2(4\xi)=\partial_{\mathrm{MV}}\nu
 \in H^1(\IP^1,\IZ)
$ for the cover
$\IP^1=B\cup\Delta_0$; see
\cite[Lemma~20.8.2 and Remark~20.13.2]{StacksProject} for these
constructions. We use local extension minus extension over $B$
as our transgression convention, with the boundary map sending
the positive meridian class to $[1]$.

Let $x_0$ be the section lift of a positive meridian around $0$.
The section extends over $B$ by Propositions~\ref{prop: construction X}
and \ref{prop: compactification X}, so $x_0$ bounds a section disc
in $Y_B$ and $u_B(x_0)=0$.
The logarithmic gluing gives $x_0^4=e_4$ modulo $\ker\xi$,
as in the proof of Proposition~\ref{prop: simply connected}.
Since $u_0$ restricts to $4\xi$ on the fibre,
\[
 4u_0(x_0)=(4\xi)(e_4)=4.
\]
Thus $\nu$ evaluates to one on the positive meridian, and
$d_2(4\xi)=\partial_{\mathrm{MV}}\nu=[1]$.\footnote{One can also see this using differential forms. Closed one-forms representing $u_0$ and $u_B$ need not agree on the overlap: up to an exact form, their difference is the pullback of a closed one-form $\eta$ on an annulus around $0$. The calculation above says that $\eta$ has integral one around the positive meridian. To obtain a two-form on the base, choose a function $\chi$ which is zero near the inner boundary of the annulus and one near the outer boundary. Then $d(\chi\eta)=d\chi\wedge\eta$, extended by zero outside the annulus, represents the transgression with our sign convention. By Stokes, its integral over the base is exactly the meridian integral of $\eta$, hence one. This also determines the integral class, since $H^2(\IP^1,\IZ)$ has no torsion.}

The remaining differentials follow from the product rule for $d_2$,
with sign given by total degree. In our description of the $E_2$-page,
products of invariant classes are wedge products, and
$ s\cup[v]=[s\wedge v]$
for an invariant class $s$ and a coinvariant class $[v]$.
Indeed, cup product by $s$ comes from the coefficient morphism
$v\mapsto s\wedge v$ \cite[Lemma~20.31.1]{StacksProject}, with which
\eqref{eq: ident cohomology} and the coinvariant identification
\cite[(5.3)]{Looijenga1992} are compatible.
All target groups below are infinite cyclic, so we may cancel
nonzero integer factors.

Write $d_2(2\omega)=m[\alpha]$ with $m\in\IZ$.
First, the product with $4\xi$ determines the next differential:
using $2\omega\wedge4\xi=4(2\omega\wedge\xi)$ and
$[\omega]=-4[\alpha\wedge\xi]$ from
\eqref{eq: invariant coinvariant relation}, we get
\[
 4d_2(2\omega\wedge\xi)
 =4m[\alpha\wedge\xi]+2[\omega]
 =4(m-2)[\alpha\wedge\xi].
\]
Thus $d_2(2\omega\wedge\xi)=(m-2)[\alpha\wedge\xi]$.
To determine $m$, apply $d_2$ to the vanishing product
$(2\omega)\wedge(2\omega\wedge\xi)=0$. This gives
\[
 0=2m[\omega\wedge\alpha\wedge\xi]
   +2(m-2)[\omega\wedge\alpha\wedge\xi],
\]
so $m=1$. This proves the two middle formulas.

Finally, $\omega^2=-4\,\omega\wedge\alpha\wedge\xi$ gives
$(2\omega)^2=-4(4\omega\wedge\alpha\wedge\xi)$. Applying $d_2$
and using $d_2(2\omega)=[\alpha]$, we obtain
\[
 -4d_2(4\omega\wedge\alpha\wedge\xi)
 =4[\omega\wedge\alpha],
\]
which proves the last formula.
\end{proof}

\begin{cor}\label{cor: integral homology}
The manifold $Y$ is an integral homology six-sphere. More precisely,
\[
 H^k(Y,\IZ)\isom H_k(Y,\IZ)\isom
 \begin{cases}
  \IZ,&k=0,6,\\
  0,&1\leq k\leq5.
 \end{cases}
\]
\end{cor}
\begin{proof}
The four differentials in Proposition~\ref{prop: Leray differentials}
are isomorphisms between the integral generators of
Proposition~\ref{prop: cohomology of direct images}. Hence
\[
 E_3^{p,q}=
 \begin{cases}
  \IZ,&(p,q)=(0,0),(2,4),\\
  0,&\text{otherwise}.
 \end{cases}
\]
There are no higher differentials and no extension problems, so
this gives the asserted cohomology groups. Since $Y$ is a smooth
compact complex threefold, it is a closed oriented real
six-manifold. Integral Poincar\'e duality gives the homology groups.
\end{proof}

\section{Conclusions}
We restate and prove the main result.
\begin{thm}\label{thm: main} Let $Y(u)$ be the family of complex manifolds constructed in the previous sections. 
	Then the differentiable manifold underlying the constructed complex manifolds $Y(u)$ is diffeomorphic to $S^6$. 
	
	No $Y(u)$ is biholomorphic to one of the  complex manifolds constructed in \cite{Engel2026, Alpoge2026}.
\end{thm}
\begin{proof}
The first assertion follows as in \cite[proof of Theorem~4.1]{Engel2026}.
By Proposition~\ref{prop: simply connected} and
Corollary~\ref{cor: integral homology}, $Y(u)$ is simply connected and has
the integral homology of $S^6$. The Hurewicz and Whitehead theorems
therefore imply that it is a homotopy six-sphere. The group
$\Theta_6$ of $h$-cobordism classes of smooth homotopy six-spheres
is trivial by \cite[\S7]{KervaireMilnor1963}. Smale's
$h$-cobordism theorem \cite{Smale1962} then implies that $Y(u)$
is diffeomorphic to the standard $S^6$.
	
	Assume $f\colon Y(u)\to Y'$ is a biholomorphism to one of the complex manifolds constructed in \cite{Engel2026}. Then $f$ induces an isomorphism on the field of meromorphic functions and hence a commutative diagram
	\[
	\begin{tikzcd}
Y(u) \rar{\isom} \dar & Y'\dar\\
\IP^1 \rar[equal] & \IP^1
	\end{tikzcd}
	\]
	because the fibrations coincide with the algebraic reduction. But the two constructions have different types of singular fibres, a contradiction.
\end{proof}

\begin{rem}
	ChatGPT \cite{openai2026chatgpt} will immediately give some information on the canonical bundle and Hodge numbers, but we leave this for later exploration.
\end{rem}

\newcommand{\etalchar}[1]{$^{#1}$}

\end{document}